\documentclass[11pt]{article}
\usepackage[utf8]{inputenc}
\usepackage{authblk}
\usepackage{mathtools}
\usepackage[english]{babel}
\usepackage{amsmath}
\usepackage{amsthm}
\usepackage{amsfonts}
\usepackage{amssymb}
\usepackage{bbm}
\usepackage{subcaption}
\usepackage{verbatim}
\usepackage[usenames,dvipsnames]{color}
\usepackage{hyperref}
\usepackage[left=2cm,right=2cm,top=2cm,bottom=2cm]{geometry}
\usepackage{dsfont}
\usepackage{enumitem}
\usepackage{tikz}
\usepackage[numeric,initials,nobysame,msc-links,abbrev]{amsrefs}
\renewcommand{\eprint}[1]{\href{https://arxiv.org/abs/#1}{arXiv:#1}}
\newcommand{\pageafter}[1]{#1~pp.}
\BibSpec{article}{%
+{} {\PrintAuthors} {author}
+{,} { \textit} {title}
+{.} { } {part}
+{:} { \textit} {subtitle}
+{,} { \PrintContributions} {contribution}
+{.} { \PrintPartials} {partial}
+{,} { } {journal}
+{} { \textbf} {volume}
+{} { \PrintDatePV} {date}
+{,} { \issuetext} {number}
+{,} { \pageafter} {pages}
+{,} { } {status}
+{,} { \PrintDOI} {doi}
+{,} { available at \eprint} {eprint}
+{} { \parenthesize} {language}
+{} { \PrintTranslation} {translation}
+{;} { \PrintReprint} {reprint}
+{.} { } {note}
+{.} {} {transition}
+{} {\SentenceSpace \PrintReviews} {review}
}
\BibSpec{collection.article}{%
+{} {\PrintAuthors} {author}
+{,} { \textit} {title}
+{.} { } {part}
+{:} { \textit} {subtitle}
+{,} { \PrintContributions} {contribution}
+{,} { \PrintConference} {conference}
+{} {\PrintBook} {book}
+{,} { } {booktitle}
+{,} { \PrintDateB} {date}
+{,} { \pageafter} {pages}
+{,} { } {status}
+{,} { \PrintDOI} {doi}
+{,} { available at \eprint} {eprint}
+{} { \parenthesize} {language}
+{} { \PrintTranslation} {translation}
+{;} { \PrintReprint} {reprint}
+{.} { } {note}
+{.} {} {transition}
+{} {\SentenceSpace \PrintReviews} {review}
}

\usetikzlibrary{arrows}
\usepackage[normalem]{ulem}

\newtheorem{theorem}{Theorem}

\newtheorem{lemma}[theorem]{Lemma}

\newtheorem{conjecture}[theorem]{Conjecture}
\newtheorem{question}[theorem]{Question}
\theoremstyle{definition}

\newtheorem{remark}[theorem]{Remark}
\newtheorem{example}[theorem]{Example}

\newcommand{\bbQ}{\mathbb{Q}}

\newcommand{\bbT}{\mathbb{T}}

\newcommand{\bbZ}{\mathbb{Z}}

\newcommand{\cD}{\mathcal{D}}

\newcommand{\bn}{\mathbf{n}}

\newcommand{\bu}{\mathbf{u}}
\newcommand{\bv}{\mathbf{v}}
\newcommand{\bw}{\mathbf{w}}
\newcommand{\bx}{\mathbf{x}}
\newcommand{\by}{\mathbf{y}}
\newcommand{\bz}{\mathbf{z}}

\title{\scshape Local dimer dynamics in higher dimensions}

\author[1]{Ivailo Hartarsky}
\author[2]{Lyuben Lichev}
\author[1]{Fabio Toninelli}
    \affil[1]{TU Wien, Faculty of Mathematics and Geoinformation,
    Wiedner Hauptstra\ss e 8-10, A-1040 Vienna, Austria, \texttt{\{ivailo.hartarsky,fabio.toninelli\}@tuwien.ac.at}}
\affil[2]{Institut Camille Jordan, Univ. Jean Monnet, Saint-Etienne, France, \texttt{lyuben.lichev@univ-st-etienne.fr}}

\begin{document}

\maketitle

\begin{abstract}
We consider local dynamics of the dimer model (perfect matchings) on hypercubic boxes $[n]^d$. These consist of successively switching the dimers along alternating cycles of prescribed (small) lengths. We study the connectivity properties of the dimer configuration space equipped with these transitions. Answering a question of Freire, Klivans, Milet and Saldanha, we show that in three dimensions any configuration admits an alternating cycle of length at most 6. We further establish that any configuration on $[n]^d$ features order $n^{d-2}$ alternating cycles of length at most $4d-2$. We also prove that the dynamics of dimer configurations on the unit hypercube of dimension $d$ is ergodic when switching alternating cycles of length at most $4d-4$. Finally, in the planar but non-bipartite case, we show that parallelogram-shaped boxes in the triangular lattice are ergodic for switching alternating cycles of lengths 4 and 6 only, thus improving a result of Kenyon and R\'emila, which also uses 8-cycles. None of our proofs make reference to height functions.
\end{abstract}


\section{Introduction}
The dimer model on planar graphs has played a crucial role in statistical
mechanics and probability theory for several reasons: in particular, its integrability
properties related to Kasteleyn's determinantal or Pfaffian solution
and, in
the bipartite case, the emergence of macroscopic shapes, arctic curves,
conformal invariance and Gaussian Free Field height fluctuations at large
scales  (see the monographs \cites{Kenyon09,Gorin21} and references therein). The behaviour of the dimer model in dimension higher than 2, or on
planar but non-bipartite graphs, is much less understood, and the same can be
said about the model's  Glauber dynamics. The goal of
the present work is to present new results about (local) dimer Glauber
dynamics, either on $\bbZ^d$ for $d \ge 3$ or on the planar
triangular lattice.

The study of Glauber dynamics of the dimer model has a long history and has
proved quite challenging. While it is easy to define local Markov
dynamics with update rule consisting in switching alternating cycles, which ensures that
the uniform measure is stationary and reversible, proving that such processes
are ergodic and quantifying their speed of convergence to equilibrium is
a much more subtle business. In the \emph{planar bipartite case}, the height
function~\cite{Thurston90} turns out to be extremely helpful: it provides a natural partial order preserved by the dynamics, an easy proof of ergodicity and an intuitive ``mean curvature motion'' heuristic suggesting that,
in many interesting situations, the mixing time $T_{\rm {mix}}$ is of order
$L^2$ (in continuous time), with $L$ the diameter of the domain. Under some conditions, the
Glauber dynamics have in fact been proven to be fast mixing~\cites{Luby01,Randall00,Wilson04} and even to satisfy
$T_{\rm {mix}} = L^{2 + o (1)}$ under suitable restrictions on the domain
geometry~\cites{Laslier15,Caputo12,Laslier23}.

As soon as the model is either not planar or not bipartite, there is no canonical definition of height function and  the most basic question of proving that local Glauber dynamics are ergodic, and even that they have no completely blocked configurations, turn out to be non-trivial. The situation is particularly unclear for the dimer model on (say, cubic subsets of)  $\mathbb Z^d$ for $d\ge 3$ where there are no local dynamics that are known to be ergodic. In fact, the simplest chain whose updates consist in flipping two parallel dimers fails to be ergodic because of subtle topological obstructions. We refer to Section~\ref{sec:background} (as well as to~\cite{Chandgotia23}*{Sections 1, 3 and 9} and~\cite{Milet15}) for a more extensive discussion of conjectures, open problems and previous partial results, and to Section~\ref{subsec:results} for a precise statement of our own results. Let us only briefly anticipate here that our main results include the proof that for local Glauber dynamics on cubic boxes of $\mathbb Z^d$, allowing \emph{switching} (also called moves in~\cite{Milet15} and loop shifts in~\cite{Chandgotia23}) along cycles of finite length (suitably depending on the dimension $d$ only), all connected components of the state space are at least of size $e^{c(d) n^{d-2}}$, with $n$ the side length of the cube. For comaprison, it was previously an open question to prove that there are no components of cardinality $1$. Let us add that the ergodicity question, besides being crucial for the use of Markov chains as simulation algorithms, has also attracted interest in the theoretical physics community due to its connection to the quantum dimer model and to the possible occurrence of ``Hilbert space fragmentation''~\cite{Roising23}. 

Finally, we emphasize that substantial progress in the understanding of the \emph{equilibrium} properties of (uniform) dimer configurations in dimension $d \ge 3$ has been made recently. This
includes a large deviation principle for the ``flow function'' of
three-dimensional dimers~\cite{Chandgotia23}, an approximate enumeration algorithm 
\cite{Kenyon96a} and the proof of occurrence of macroscopic loops for the $d \ge 3$ double dimer model~\cite{Quitmann23}. See also~\cites{Lammers21,Randall00a,Linde01} for different generalisations of the dimer model to dimension $d \ge 3$.

\subsection{Model}
Given a graph $G$, a \emph{dimer configuration} on $G$ is a perfect matching, that is, a set of edges such that every vertex in $G$ is incident to exactly one of them. The edges in a dimer configuration are called \emph{dimers}. 
Fix a graph $G$ and a dimer configuration on $G$. An \emph{alternating cycle} is a cycle in $G$ of even length where every second edge is a dimer. A \emph{switching} of an alternating cycle is the operation of exchanging the dimer and the non-dimer edges along the cycle. Note that any switching in a dimer configuration produces another dimer configuration, see e.g.\ Figure~\ref{fig:8cyc}.

For a graph $G$, we denote by $\cD(G)$ the graph with vertices given by the dimer configurations on $G$, where two vertices are connected if there is an alternating cycle whose switching transforms one of the dimer configurations into the other. Moreover, for an integer $\ell\ge 2$, we denote by $\cD_{\ell}(G)$ the spanning subgraph of $\cD(G)$ where edges correspond to alternating cycles of length at most $2\ell$. We also say that the space of dimer configurations is \emph{$2\ell$-ergodic} (or simply \emph{ergodic}) if the graph $\cD_{\ell}(G)$ is connected. Note that the superposition of two dimer configurations forms a set of alternating cycles and double edges, so for any finite graph $G$ and $\ell$ large enough, $\cD_\ell(G)$ is necessarily ergodic.

Given a positive integer $n$, we denote by $[n]$ the set $\{1,\dots,n\}$. Given a positive integer $d\ge 1$, we refer to any vector $\bn=(n_1,\dots,n_d)$ such that $n_1,\dots,n_d\ge2$ and {the product} $n_1\cdots n_d$ is even as \emph{shape}. For a shape $\bn$, the \emph{$\bn$-box} $\bbQ^d_\bn$ is defined as follows. The graph $\bbQ_\bn^d$ has vertex set $\prod_{i=1}^d[n_i]$ and edges between $\bu = (u_1,\ldots,u_d)$ and $\bv = (v_1,\ldots,v_d)$ if there exists $i\in[d]$ such that $u_j=v_j$ for all $j\neq i$ and $|u_i-v_i|=1$. We write $\bbQ_n^d$ for $\bbQ_{\bn}^d$ with $\bn=(n,\dots,n)\in\bbZ^d$ and $\bbQ^d$ for the unit hypercube $\bbQ_2^d$. For simplicity, we often identify boxes $\bbQ_\bn^d$ with their natural embedding in the $d$-dimensional Euclidean space.

We further define the triangular lattice $\bbT$ as the graph with vertex set $\mathbb Z^2$ where every vertex $\bv$ is adjacent to $\bv+(1,0), \bv+(-1,0), \bv+(0,1), \bv+(0,-1), \bv+(1,-1), \bv+(-1,1)$. For positive integers $m,n$ with $mn$ even, we denote by $\bbT_{m,n}$ the graph induced from $\bbT$ by the vertex set $[m]\times [n]$. Note that, while $\bbT_{m,n}$ is a (rectangular) box in the embedding of $\bbT$ chosen above, in the more standard isoradial embedding, these domains correspond to parallelograms.

We set out to study the ergodicity of $\cD_\ell(\bbQ_\bn^d)$ and $\cD_{\ell}(\bbT_{m,n})$.

\subsection{Background}
\label{sec:background}
Given a planar graph drawn in the plane, a \emph{domain} is a union of faces (seen as closed polygons in the plane) of some fixed lattice. Of course, any domain may be seen as a portion of the lattice with its proper dimer configurations and cycle-switching dynamics. The ergodicity of simply connected domains in planar bipartite lattices (with cycle length given by the largest number of edges on the boundary of an inner face) is a classical fact and follows directly by considering the associated height function (see~\cites{Saldanha95,Thurston90}). Recently, the ergodicity of local dynamics on a number of planar lattices was studied by R{\o}ising and Zhang~\cite{Roising23}, extending an approach of  Kenyon and  R\'emila~\cite{Kenyon96}. We note that the techniques used there do not rely on height functions but use planarity in a substantial way, and also involve a certain amount of manual verification.

The main goal of our work is to go beyond the planar case. The most natural setting in this respect corresponds to studying $\cD_\ell(\bbQ_\bn^3)$ for fixed $\ell$ and large 3-dimensional shapes $\bn$. It is not hard to check that $\cD_2(\bbQ_\bn^3)$ is not connected and even has isolated vertices e.g.\ for $\bn=(3,3,2)$ (and similarly for $\bn=(n_1,n_2,n_3)$ with $n_1$ and $n_2$ divisible by 3 and $n_3$ even), see Figure~\ref{fig:332}. This suggests the existence of an invariant preserved by switching $4$-cycles. One such invariant was noted in~\cites{Freedman11} (also see \cite{Bednik19}), and a much more informative one called the \emph{twist} was introduced in~\cite{Milet15} and further studied in~\cites{Milet15a,Milet14}, thus establishing interesting algebraic, topological and geometric connections alongside the combinatorial ones.

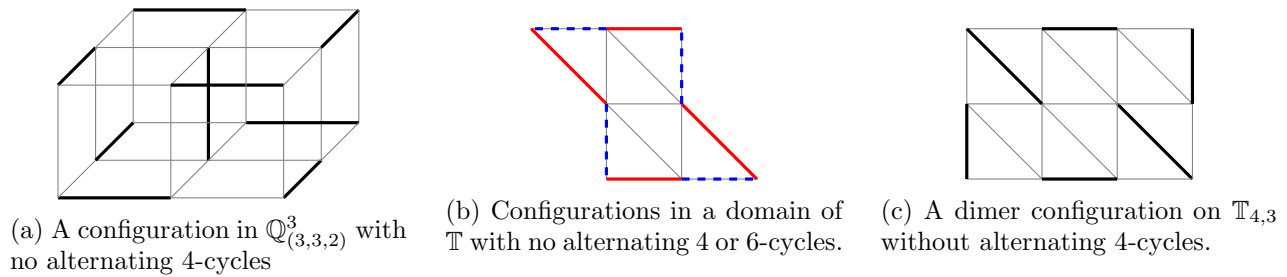
\begin{figure}
\centering
\begin{subfigure}{0.3\textwidth}
\centering
\begin{tikzpicture}[x=0.5cm,y=0.5cm]
\draw [very thick] (-3,-1)-- (0,-1);
\draw [gray] (0,-1)-- (3,-1);
\draw [very thick] (3,-1)-- (4,0);
\draw [gray] (0,-1)-- (1,0);
\draw [gray] (-3,-1)-- (-2,0);
\draw [gray] (-2,0)-- (1,0);
\draw [gray] (1,0)-- (4,0);
\draw [gray] (-3,-1)-- (-3,2);
\draw [very thick] (-3,2)-- (-2,3);
\draw [gray] (-2,3)-- (1,3);
\draw [very thick] (1,3)-- (1,0);
\draw [gray] (-3,2)-- (0,2);
\draw [gray] (0,2)-- (1,3);
\draw [gray] (0,2)-- (0,-1);
\draw [gray] (-2,3)-- (-2,0);
\draw [gray] (1,3)-- (4,3);
\draw [gray] (4,3)-- (4,0);
\draw [very thick] (0,2)-- (3,2);
\draw [gray] (3,2)-- (4,3);
\draw [gray] (3,2)-- (3,-1);
\draw [very thick] (4,3)-- (5,4);
\draw [gray] (1,3)-- (2,4);
\draw [very thick] (-1,4)-- (2,4);
\draw [gray] (2,4)-- (5,4);
\draw [gray] (5,4)-- (5,1);
\draw [very thick] (5,1)-- (2,1);
\draw [gray] (2,1)-- (-1,1);
\draw [gray] (-1,1)-- (-1,4);
\draw [gray] (-1,4)-- (-2,3);
\draw [very thick] (-2,0)-- (-1,1);
\draw [gray] (1,0)-- (2,1);
\draw [gray] (2,1)-- (2,4);
\draw [gray] (4,0)-- (5,1);
\end{tikzpicture}
\caption{A configuration in $\bbQ^{3}_{(3,3,2)}$ with no alternating 4-cycles}
\label{fig:332}
\end{subfigure}
\quad
\begin{subfigure}{0.3\textwidth}
\centering
\begin{tikzpicture}[x=1cm,y=1cm]
        \draw[gray,very thin] (0,1)--(1,0)--(1,1)--(0,1)--(0,2)--(1,1);
        \draw[very thick,blue] (0,0)--(1,0);
        \draw[very thick,blue] (2,0)--(1,1);
        \draw[very thick,blue] (1,2)--(0,2);
        \draw[very thick,blue] (-1,2)--(0,1);
        \draw[very thick,red,dashed] (1,0)--(2,0);
        \draw[very thick,red,dashed] (1,1)--(1,2);
        \draw[very thick,red,dashed] (0,2)--(-1,2);
        \draw[very thick,red,dashed] (0,1)--(0,0);
    \end{tikzpicture}
    \caption{Configurations in a domain of $\bbT$ with no alternating $4$ or $6$-cycles.\label{fig:8cyc}}
\end{subfigure}
\quad
\begin{subfigure}{0.3\textwidth}
    \centering
\begin{tikzpicture}[x=1cm,y=1cm]
\draw [very thick] (5,1)-- (6,0);
\draw [gray] (5,1)-- (5,0);
\draw [gray] (5,1)-- (5,2);

\draw [gray] (4,1)-- (4,0);
\draw [gray] (4,1)-- (4,2);

\draw [gray] (6,1)-- (6,0);
\draw [gray] (3,1)-- (3,2);

\draw [gray] (5,0)-- (6,0);
\draw [gray] (3,0)-- (4,0);

\draw [gray] (5,1)-- (6,1);
\draw [gray] (3,1)-- (4,1);

\draw [gray] (5,2)-- (6,2);
\draw [gray] (3,2)-- (4,2);

\draw [gray] (4,1)-- (5,1);
\draw [gray] (5,0)-- (4,1);
\draw [gray] (5,1)-- (4,2);

\draw [gray] (6,1)-- (5,2);
\draw [gray] (4,0)-- (3,1);

\draw [very thick] (6,1)-- (6,2);
\draw [very thick] (5,2)-- (4,2);
\draw [very thick] (3,2)-- (4,1);
\draw [very thick] (5,0)-- (4,0);
\draw [very thick] (3,1)-- (3,0);
\end{tikzpicture}
\caption{A dimer configuration on $\bbT_{4,3}$ without alternating 4-cycles.}
\label{fig:no4cycs}
\end{subfigure}
\caption{Example configurations with no short alternating cycles.}
\end{figure}

This suggests considering $\cD_\ell(\bbQ_\bn^3)$ for $\ell\ge3$. Milet and Saldanha~\cite{Milet15a} asked whether $\cD_3(\bbQ_\bn^3)$ is connected for $3$-dimensional shapes $\bn$. This was reiterated by Freire, Klivans, Milet and Saldanha in~\cites{Freire22,Saldanha21} and very recently by Chandgotia, Sheffield and Wolfram \cite{Chandgotia23}*{Problem 9.1}. We promote it to the following conjecture, which is one of the main motivations behind our work.
\begin{conjecture}
\label{conj:ergodicity}
For all $3$-dimensional shapes $\bn$, the graph $\cD_3(\bbQ_\bn^3)$ is connected.
\end{conjecture}
Several weaker results in the direction of Conjecture~\ref{conj:ergodicity} have been obtained. Firstly, for $\bn=(n_1,n_2,2)$, $6$-ergodicity was established in~\cite{Milet15}. In~\cite{Freire22}, Conjecture~\ref{conj:ergodicity} was proved up to refinement, that is, repeatedly replacing each dimer in the configuration by a copy of a dimer configuration on $\bbQ_{(5,5,10)}$, $\bbQ_{(5,10,5)}$ or $\bbQ_{(10,5,5)}$ with dimers parallel to the original one. In~\cite{Saldanha22}, Conjecture~\ref{conj:ergodicity} was proved for $\bn=(n,m,N)$ with $N$ large enough (depending on $n$ and $m$) and restricting attention to dimer configurations whose last sufficiently many layers (again, depending on $n$ and $m$) are filled with vertical dimers. 

In view of the above, the following question weakening Conjecture~\ref{conj:ergodicity} was asked in~\cite{Freire22}{, where it was checked that no small counterexamples exist}.
\begin{question}
\label{ques:isolated}
Do there exist even $n$ such that $\cD_3(\bbQ_n^3)$ has isolated vertices?
\end{question}

Higher dimensions have been explored even less. Indeed, we are only aware of a binary invariant for $4$-cycle switchings considered in~\cite{Klivans22}, where results similar to those from~\cite{Saldanha22} were proved.

\subsection{Results}
\label{subsec:results}
In the present work, we prove several results on the connectivity properties of the graphs $\cD_\ell(\bbQ^d_\bn)$ and $\cD_\ell(\bbT_{m,n})$. Contrary to previous approaches that mainly used the algebraic, topological and geometric aspects of the Milet--Saldanha twist invariant, our arguments are purely combinatorial and elementary. 

Our first result immediately entails a negative answer to Question~\ref{ques:isolated}.
\begin{theorem}[Extraction of a dense $\bbQ^d$]
\label{th:counting}
Let $d\ge 2$ and $\bn$ be a $d$-dimensional shape. Then, for any dimer configuration $D$ on $\bbQ_\bn^d$, there exists $\bx\in\bbZ^d$ such that the unit cube $\bx+\bbQ^d\subseteq\bbQ_\bn^d$ contains at least $2^{d-2}+1$ dimers in $D$.
\end{theorem}
Indeed, when $d=3$, this yields a unit cube with $3$ dimers, which is readily checked to contain an alternating cycle of length $4$ or $6$. One can similarly check that in 4 dimensions, we obtain an alternating cycle of length at most $8$. We believe that any set of $2^{d-2}+1$ disjoint dimers in $\bbQ^d$ admits an alternating cycle of length at most $2d$ for any $d\ge 2$ but have been unable to prove this. Let us emphasise that it is crucial that $\bx+\bbQ^d$ contains $2^{d-2}+1$ dimers and not less: in fact, one can check that there exist various very different examples of configurations on $\bbQ^d$ with $2^{d-2}$ dimers containing no alternating cycle of any length.

Theorem~\ref{th:counting} proves the absence of isolated vertices in $\cD_3(\bbQ_\bn^3)$ for all $3$-dimensional shapes $\bn$. At the cost of increasing the value of $\ell$, we are able to prove much more.
\begin{theorem}[Degree and component size]
\label{th:degree:size}
Fix $d\ge 3$ and an even positive integer $n$. The minimum degree of $\cD_{2d-1}(\bbQ_n^d)$ is at least $n^{d-2}/(320d^6)$ and each connected component contains at least $2^{n^{d-2}/(320d^6)}$ dimer configurations. 
\end{theorem}
Let us note that we have specialised the result for boxes of equal sides only for the sake of readability. The bound on the minimum degree is optimal up to a factor independent of $n$, as shown by the pyramid configuration defined in Example~\ref{ex:pyramids} where alternating cycles of length $\ell$ should stay at distance at most $\ell$ from the centers of the $(d-1)$-dimensional horizontal sections, see Figure~\ref{fig:pyramids}. We further remark that the total number of dimer configurations on $\bbQ_n^d$ is of order $e^{C(d)n^d}$ for some $C(d)>0$ as $n$ grows~\cite{Hammersley66}.

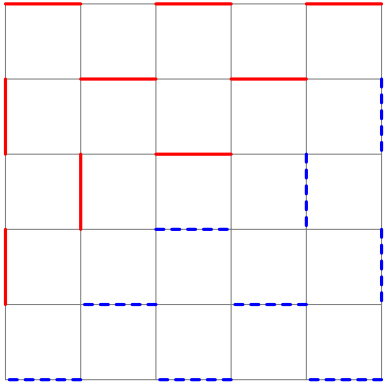
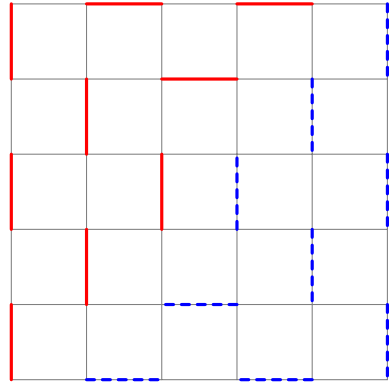
\begin{figure}
\centering
\begin{subfigure}{0.45\textwidth}
\centering
\begin{tikzpicture}[line cap=round,line join=round,x=1.0cm,y=1.0cm]
\draw[gray,very thin] (-2,-2)--(-2,-1)--(-1,-1)--(-1,-2)--(0,-2)--(0,-1)--(1,-1)--(1,-2)--(2,-2)--(2,-1)--(3,-1)--(3,-2);
\draw[gray,very thin] (-1,-1)--(-1,0)--(-2,0)--(-2,1)--(-1,1)--(-1,2)--(-2,2)--(-2,3);
\draw[gray,very thin] (-1,2)--(-1,3)--(0,3)--(0,2)--(1,2)--(1,3)--(2,3)--(2,2)--(3,2)--(3,3);
\draw[gray,very thin] (2,2)--(2,1)--(3,1)--(3,0)--(2,0)--(2,-1);
\draw[gray,very thin] (1,-1)--(1,0)--(2,0);
\draw[gray,very thin] (0,-1)--(0,0)--(-1,0);
\draw[gray,very thin] (0,2)--(0,1)--(-1,1);
\draw[gray,very thin] (1,2)--(1,1)--(2,1);
\draw[gray,very thin] (1,0)--(1,1);
\draw[gray,very thin] (0,0)--(0,1);
\draw[very thick,blue,dashed] (0,0)-- (1,0);
\draw[very thick,red] (1,1)-- (0,1);
\draw[very thick,red] (-1,2)-- (0,2);
\draw[very thick,red] (1,2)-- (2,2);
\draw[very thick,blue,dashed] (2,1)-- (2,0);
\draw[very thick,blue,dashed] (2,-1)-- (1,-1);
\draw[very thick,blue,dashed] (0,-1)-- (-1,-1);
\draw[very thick,red] (-1,0)-- (-1,1);
\draw[very thick,red] (-2,2)-- (-2,1);
\draw[very thick,red] (-2,3)-- (-1,3);
\draw[very thick,red] (0,3)-- (1,3);
\draw[very thick,red] (2,3)-- (3,3);
\draw[very thick,blue,dashed] (3,2)-- (3,1);
\draw[very thick,blue,dashed] (3,0)-- (3,-1);
\draw[very thick,blue,dashed] (3,-2)-- (2,-2);
\draw[very thick,blue,dashed] (1,-2)-- (0,-2);
\draw[very thick,blue,dashed] (-1,-2)-- (-2,-2);
\draw[very thick,red] (-2,-1)-- (-2,0);
\end{tikzpicture}
\caption{The dimer configuration on odd sections.\label{subfig:left}}
\end{subfigure}
\begin{subfigure}{0.45\textwidth}
\centering
\begin{tikzpicture}[line cap=round,line join=round,x=-1.0cm,y=1.0cm,rotate=90]    
\draw[gray,very thin] (-2,-2)--(-2,-1)--(-1,-1)--(-1,-2)--(0,-2)--(0,-1)--(1,-1)--(1,-2)--(2,-2)--(2,-1)--(3,-1)--(3,-2);
\draw[gray,very thin] (-1,-1)--(-1,0)--(-2,0)--(-2,1)--(-1,1)--(-1,2)--(-2,2)--(-2,3);
\draw[gray,very thin] (-1,2)--(-1,3)--(0,3)--(0,2)--(1,2)--(1,3)--(2,3)--(2,2)--(3,2)--(3,3);
\draw[gray,very thin] (2,2)--(2,1)--(3,1)--(3,0)--(2,0)--(2,-1);
\draw[gray,very thin] (1,-1)--(1,0)--(2,0);
\draw[gray,very thin] (0,-1)--(0,0)--(-1,0);
\draw[gray,very thin] (0,2)--(0,1)--(-1,1);
\draw[gray,very thin] (1,2)--(1,1)--(2,1);
\draw[gray,very thin] (1,0)--(1,1);
\draw[gray,very thin] (0,0)--(0,1);
\draw[very thick,blue,dashed] (0,0)-- (1,0);
\draw[very thick,red] (1,1)-- (0,1);
\draw[very thick,red] (-1,2)-- (0,2);
\draw[very thick,red] (1,2)-- (2,2);
\draw[very thick,blue,dashed] (2,1)-- (2,0);
\draw[very thick,blue,dashed] (2,-1)-- (1,-1);
\draw[very thick,blue,dashed] (0,-1)-- (-1,-1);
\draw[very thick,red] (-1,0)-- (-1,1);
\draw[very thick,red] (-2,2)-- (-2,1);
\draw[very thick,red] (-2,3)-- (-1,3);
\draw[very thick,red] (0,3)-- (1,3);
\draw[very thick,red] (2,3)-- (3,3);
\draw[very thick,blue,dashed] (3,2)-- (3,1);
\draw[very thick,blue,dashed] (3,0)-- (3,-1);
\draw[very thick,blue,dashed] (3,-2)-- (2,-2);
\draw[very thick,blue,dashed] (1,-2)-- (0,-2);
\draw[very thick,blue,dashed] (-1,-2)-- (-2,-2);
\draw[very thick,red] (-2,-1)-- (-2,0);
\end{tikzpicture}
\caption{The dimer configuration on even sections.\label{subfig:right}}
\end{subfigure}
    \caption{The pyramid dimer configuration on $\bbQ_n^3$ (see Example~\ref{ex:pyramids}; the parity of a section is determined by the parity of the vertex $(1,1,\bv')$). The horizontal sections alternate between the two 2-dimensional dimer configurations shown above. These contain no vertical dimers. The only short alternating cycles are around the middle column. Red dimers are solid, blue ones are dashed.}
    \label{fig:pyramids}
\end{figure}

Next, we focus our attention on unit hypercubes $\bbQ^d$ for which we can prove stronger results. Firstly, the following improvement of Theorem~\ref{th:degree:size} can be deduced in the same way.
\begin{theorem}[Degree and component size in $\bbQ^d$]
\label{th:degree:size:hypercube}
    Fix $d\ge3$. The minimum degree of $\cD_{d-1}(\bbQ^d)$ is at least $2^{d}/d^4$ and each connected component contains at least $2^{2^{d}/d^4}$ dimer configurations. 
\end{theorem}
While this result is only interesting for large $d$, at which point the size of alternating cycles also grows, we emphasise that the size of cycles we allow is very small compared to the total volume of the hypercube. At the cost of allowing switching along cycles twice as long, we are able to prove ergodicity in the following strong form.
\begin{theorem}[Ergodicity on $\bbQ^d$]
\label{th:ergodicity:hypercube}
For every $d\ge 2$, the graph $\cD_{2d-2}(\bbQ^d)$ is connected and has diameter at most $(d-1)2^{d-1}$.
\end{theorem}

In view of this bound on the diameter of the dimer configuration space on the unit hypercube that is almost linear (as a function of the volume of the hypercube), it is natural to ask whether the diameter of $\cD_{\ell}(\bbQ_n^d)$ can also be linear in the volume. In two dimensions, this is not the case, since a lower bound of order $n^3$ follows by considering the height functions of the two configurations in Figure~\ref{fig:pyramids}, while for non-bipartite graphs such as the triangular lattice, the diameter can be linear in the volume~\cite{Kenyon96}. In higher dimensions, the situation is only clarified by the following result, which also offers a proof in two dimensions without height functions.
\begin{theorem}[Diameter lower bound]
\label{th:diameter}
For all $d,\ell,n\ge 2$ with $n$ even, the graph $\cD_{\ell}(\bbQ_n^d)$ has diameter at least \[\frac{n^{d-1}(n^2-1)}{6\ell^2}.\]
\end{theorem}

Finally, while it is tempting to look for the minimal value of $\ell$ ensuring ergodicity in each setting, this is a rather sensitive matter. In this direction, we prove the following result.
\begin{theorem}[Ergodicity on $\bbT_{m,n}$]
    \label{th:triangular}
    For all positive integers $m,n$ with $mn$ even, the graph $\cD_3(\bbT_{m,n})$ is connected and has diameter at most $2mn$.
\end{theorem}
This should be compared to~\cite{Kenyon96} showing an analogous result for $8$-ergodicity but for any simply connected domain of the triangular lattice. While there do exist domains for which 8-cycles are necessary (see Figure~\ref{fig:8cyc}), we show that for boxes, cycles of length 4 and 6 suffice. While there is a dimer configuration on $\bbT_{4,3}$ without alternating 4-cycles (see Figure~\ref{fig:no4cycs}), it remains unclear whether 6-cycles can be avoided for sufficiently large values of $mn$.

With the exception of Theorems~\ref{th:degree:size} and~\ref{th:degree:size:hypercube}, the above results are shown by completely different means, thus providing several approaches for further use. The proofs provided in subsequent sections are therefore completely independent.

\section{Extraction of a dense \texorpdfstring{$\bbQ^d$}{Qd}: proof of Theorem~\ref{th:counting}}

In this section, we show that any dimer configuration on $\bbQ_\bn^d$ contains a unit hypercube containing at least $2^{d-2}+1$ dimers. The proof is a simple double counting argument.

\begin{proof}[Proof of Theorem~\ref{th:counting}]
Fix a dimension $d\ge 2$, a $d$-dimensional shape $\bn$ and a dimer configuration $D$ on $\bbQ_\bn^d$. 
Let $\Pi_{\bn}^d = \{0,\dots,n_1-2\}\times\dots\times\{0,\dots,n_d-2\}$. We count the couples $(\bu,\bx)\in\bbQ_\bn^d\times\Pi_\bn^d$ such that the dimer of $\bu$ in $D$ is contained in the unit hypercube $\bx+\bbQ^d\subseteq \bbQ^d_\bn$. We say that a vertex $\bw\in\bbQ_\bn^d$ is of \emph{type} $I(\bw)=\{i\in[d]:w_i\in\{1,n_i\}\}$, which indicates in which coordinates $\bw$ is on the boundary of $\bbQ^d_\bn$. The type of a dimer $\bu\bv$ is defined as $I(\bu\bv)=\{i\in[d]:u_i=v_i\in\{1,n_i\}\}$ so that $I(\bu\bv)\subseteq I(\bu)\cap I(\bv)$. Then, for every $i\in I(\bu\bv)$, all unit hypercubes $\bx+\bbQ^d\subseteq\bbQ_\bn^d$ that contain $\bu\bv$ contain only vertices whose $i$-th component is either in the set $\{1,2\}$ (if $u_i = v_i = 1$) or in the set $\{n_i-1,n_i\}$ (if $u_i = v_i = n_i$). One may easily check that there are exactly $2^{d-|I(\bu\bv)|-1}\ge 2^{d-|I(\bu)|-1}$ such unit hypercubes.

On the other hand, the number of vertices of type $I\subseteq [d]$ is $2^{|I|}\prod_{i\in[d]\setminus I}(n_i-2)$. Treating separately the corners of $\bbQ_\bn^d$, each of which is contained in a single unit hypercube, we obtain that the number of couples $(\bu,\bx)$ as above is at least
\[2^d+\sum_{I\subsetneq [d]}2^{d-|I|-1}\times2^{|I|}\prod_{i\in[d]\setminus I}(n_i-2)=2^{d-1}\left(1+\sum_{I\subseteq[d]}\prod_{i\in[d]\setminus I}(n_i-2)\right)=2^{d-1}\left(1+\prod_{i\in[d]}(n_i-1)\right).\]
Since $|\Pi_{\bn}^d| = \prod_{i\in [d]} (n_i-1)$ and each dimer is counted twice for every unit hypercube that contains it (once for each of its endvertices), there has to be a unit hypercube containing at least $2^{d-2}+1$ dimers, as desired.
\end{proof}

\section{Degree and component size: proof of Theorem~\ref{th:degree:size} and~\ref{th:degree:size:hypercube}}
\label{sec:degree}

Throughout this section, we fix $d\ge 3$, an even positive integer $n$ and a dimer configuration $D$ on $\bbQ^d_n$. We say that a vertex $\bw\in\bbQ_n^d$ is \emph{forbidden} (by a dimer $\bu\bv\in D$) if there exists $i\in[d]$ such that $u_j=v_j=w_j$ for all $j\in[d]\setminus\{i\}$, $\bu\bv$ is a dimer and $\bu\bw$ or $\bv\bw$ is an edge in $\bbQ^d_n$. In other words, the dimer $\bu\bv$ forbids $\bw$ if $\bw$ is a neighbour of $\bu$ or $\bv$ aligned with $\bu\bv$. A vertex that is not forbidden by any dimer in $D$ is called \emph{authorised}. Observe that for $n\ge 3$, a dimer containing a vertex on the boundary of $\bbQ^d_n$ (that is, one with less than $2d$ neighbours) may forbid one or two vertices in $\bbQ^d_n$, while all other dimers forbid two vertices.

We next show that there is a short alternating cycle close to each authorised vertex.

\begin{lemma}\label{lem:authorised}
For any authorised vertex $\bw\in\bbQ_n^d$, there is an alternating cycle of length at most $4d-2$ contained in the second neighbourhood of $\bw$ in $\bbQ^d_n$. Moreover, if $n=2$, then every vertex in $\bbQ^d$ has an alternating cycle of length at most $2d-2$ in its second neighbourhood.
\end{lemma}
\begin{proof}
Let $(\bv^i)_{i=1}^{m}$ be the neighbours of the authorised vertex $\bw$ in $\bbQ^d_n$ for some $m\in [d,2d]$. Moreover, suppose without loss of generality that the dimers in the second neighbourhood of $\bw$ are $\bw \bv^1$ and $(\bu^i \bv^i)_{i=2}^{m}$.

If $\bu^i$ is a neighbour of $\bv^1$ in $\bbQ^d_n$ for some $i\in [2,m]$, then $\bv^1, \bw, \bv^i, \bu^i$ is an alternating cycle of length $4$, as desired. Suppose that none of $(\bu^i)_{i=2}^{m}$ is a neighbour of $\bv^1$. Note that since $\bw$ is an authorised vertex, the vectors $\bw-\bv^i$ and $\bv^i-\bu^i$ are orthogonal for every $i\in [2,m]$, which implies  that $\bu^i$ has exactly two neighbours among $(\bv^j)_{j=2}^{m}$ in $\bbQ^d_n$ (or otherwise said, vertex $(1,1,0,0,\dots,0)$ has exactly two neighbours in common with the origin in $\bbZ^d$). Now, consider the graph induced from $\bbQ^d_n$ by the vertices $\bigcup_{i=2}^{m} \{\bu^i, \bv^i\}$, and orient the dimer edges from $(\bv^i)_{i=2}^{m}$ to $(\bu^i)_{i=2}^m$ and the other edges from $(\bu^i)_{i=2}^m$ to $(\bv^i)_{i=2}^m$. We obtain a digraph on $2m-2$ vertices where every vertex has out-degree at least 1, and for every (directed) path in this digraph, exactly one of every two consecutive edges is a dimer. Thus, one may find an alternating cycle of length at most $2m-2\le 4d-2$ by starting from any vertex and making steps in the digraph until the first time a vertex is visited twice.

It remains to notice that, when $n=2$, every vertex is authorised and has degree $d$ in $\bbQ^d$.
\end{proof}

Lemma~\ref{lem:authorised} implies that if there are many authorised vertices in a dimer configuration on $\bbQ^d_n$, then there must be many alternating cycles in that configuration. The next lemma shows that there must be many authorised vertices.

Our argument goes roughly as follows. We may assume that a positive fraction of the dimers point in the ``vertical'' $d$-th direction. Since each dimer forbids at most two vertices, each vertex is forbidden at most once on average and no more than $2d$ times in total. Therefore, for every $\ell\ge 1$, if either at least $\ell$ vertices are forbidden more than once or at least $\ell$ dimers only forbid a single vertex, then there must be at least $\ell$ authorised vertices.

We partition the vertices of $\bbQ_n$ into ``levels'', according on their vertical coordinate. We next consider the number $k_i$ of dimers with endpoints on levels $i-1$ and $i$. Vertical dimers have two effects on the vertices at level $i$: they forbid $k_{i-1}+k_{i+2}$ vertices (with multiplicity), but they occupy $k_i+k_{i+1}$ vertices on level $i$, which cannot contribute to forbidding any vertex on that level. This gives rise to the expression $\Delta_i=((k_{i+2}-k_{i+1})-(k_i-k_{i-1}))/2$, reminiscent of a discrete second derivative. We show that, if $|\Delta_i|$ is large, then, at level $i$, there are many
\begin{itemize}
\item dimers forbidding only one vertex, whose number is denoted by $s_i$, or
\item vertices forbidden by zero, two or more dimers, whose number is denoted by $N_i$.
\end{itemize}
Since we already established that $\sum_i (s_i+N_i)$ is small, the sequence $k_i$ is approximately a one-dimensional harmonic function and, therefore, approximately linear. Recalling that $\sum_ik_i$ is a positive fraction of the volume of $\bbQ_n^d$, this implies that the boundary values $k_2$ and $k_{n}$ cannot both be small. But these boundary vertical dimers only forbid a single vertex, which implies that there are many authorised vertices, as claimed above.

\begin{lemma}\label{lem:manyvert}
Fix $n\ge 4$. Then, there are at least $n^{d-2}/(20d^2)$ authorised vertices.
\end{lemma}
\begin{proof}
By the pigeonhole principle, there exists $j\in [d]$ such that at least $n^d/(2d)$ of all dimers have endpoints which differ in their $j$-th coordinate. We assume that $j = d$ without loss of generality. For every $\bv\in\bbQ_n^d$, we define the \emph{level} of $\bv$ to be $v_d\in [n]$. For every $i\in [2,n]$, denote by $k_i$ the number of dimers $\bu\bv$ with $u_d = i-1$ and $v_d = i$. In particular,
\begin{equation}\label{eq:sum k-s}
\sum_{i=2}^n k_i \ge \frac{n^d}{2d}.
\end{equation}
For convenience of notation, we extend the sequence by setting $k_0 = k_1 = k_{n+1} = k_{n+2} = 0$.

For every $i\in [n]$, denote by $N_i$ the number of vertices on level $i$ that are either forbidden by at least two dimers or authorised. To begin with, we show that the sum of $(N_i)_{i=1}^n$ is at most twice the number of all authorised vertices. For every $j\in \{0,\ldots,2d\}$, denote by $F_j$ the number of vertices forbidden by exactly $j$ dimers. Since each of the $n^d/2$ dimers in the configuration forbids at most two vertices, we have
\[F_1 - 2F_0 + 2\sum_{i=1}^{n} N_i = F_1 + \sum_{j=2}^{2d} 2F_j\le \sum_{j=1}^{2d} jF_j\le n^d = \sum_{j=0}^{2d} F_j = F_1 + \sum_{i=1}^n N_i,\]
which implies that $\sum_{i=1}^{n} N_i\le 2F_0$. In particular, if $\sum_{i=1}^{n} N_i\ge n^{d-2}/(10d^2)$, then the proof is completed.

We focus on the analysis of the sequence $(N_i)_{i=1}^n$. Suppose that $\sum_{i=1}^{n} N_i < n^{d-2}/(10d^2)$. Observe that for every $i\in [n]$, every dimer $\bu\bv$ satisfying $u_d+1 = v_d\in \{i-1, i+2\}$ forbids one vertex on level $i$ but has no vertex on level $i$ itself. Conversely, every dimer $\bu\bv$ satisfying $u_d+1 = v_d\in \{i,i+1\}$ contains one vertex on level $i$ while forbidding none. Moreover, the dimers with two vertices on level $i$ forbid one or two vertices on that level. Let $s_i$ be the number of such dimers forbidding one vertex on level $i$, and $t_i$ be the number of such dimers forbidding two vertices on level $i$. Then, the total number of forbidden vertices on level $i$ counted with multiplicities (that is, a vertex forbidden by $j$ dimers is counted $j$ times) is $s_i + 2t_i + k_{i-1} + k_{i+2}$. Note that this expression may be rewritten as
\begin{align}
2(s_i+t_i) + k_{i-1} + k_{i+2} - s_i 
={}& (n^{d-1}-k_i-k_{i+1}) + k_{i-1} + k_{i+2} - s_i\nonumber\\
={}& n^{d-1} - (s_i-k_{i-1}+k_i+k_{i+1}-k_{i+2}).\label{eq:forbidden}
\end{align}
As a consequence, the total number of authorised vertices satisfies
\[F_0=n^d-\sum_{j=1}^{2d}F_j\ge n^d - \sum_{i=1}^n (n^{d-1} - (s_i-k_{i-1}+k_i+k_{i+1}-k_{i+2})) = k_2 + k_n+\sum_{i=1}^n s_i.\]
Hence, if $\sum_{i=1}^n s_i\ge n^{d-2}/(20d^2)$, then the proof is completed.

Suppose that $\sum_{i=1}^n s_i < n^{d-2}/(20d^2)$. Since the minimum of~\eqref{eq:forbidden} and $n^{d-1}$
is an upper bound on the number of vertices on level $i$ forbidden by exactly one dimer and no vertex is forbidden more than $2d$ times, we have that for every $i\in [n]$,
\[2dN_i \ge |s_i-k_{i-1}+k_i+k_{i+1}-k_{i+2}|.\]
In particular, combining this with the triangle inequality implies that for every $i\in [n]$,
\begin{equation}\label{eq:triangle ineq}
||k_{i+2}-k_{i+1}|-|k_i-k_{i-1}||\le |k_{i-1}-k_i-k_{i+1}+k_{i+2}|\le s_i + 2dN_i.
\end{equation}

Now, for every positive integer $m\le n/2$, summing~\eqref{eq:triangle ineq} for $i\in \{2,4,\ldots,2m-2\}$ and applying the triangle inequality yields
\begin{equation}\label{eq:sumi1}
||k_{2m} - k_{2m-1}| - k_2|=||k_{2m} - k_{2m-1}| - |k_2-k_1||  \le \sum_{i=1}^{m-1} (s_{2i}+2dN_{2i}) < \frac{3n^{d-2}}{10d},
\end{equation}
while summing over $i\in \{2m-1,2m+1,\ldots,n-1\}$ instead yields
\begin{equation}\label{eq:sumi2}
|k_n - |k_{2m-1}-k_{2m-2}||=||k_{n+1} - k_n| - |k_{2m-1}-k_{2m-2}||\le \sum_{i=m}^{n/2} (s_{2i-1}+2dN_{2i-1}) < \frac{3n^{d-2}}{10d}.
\end{equation}

To finish the proof, we show that $k_2\ge n^{d-2}/(10d)$ or $k_n\ge n^{d-2}/(10d)$. Indeed, in this case, there are at least $n^{d-2}/(10d)$ dimers that forbid only one vertex, which is a lower bound on the number of authorised vertices. Suppose for a contradiction that both $k_2 < n^{d-2}/(10d)$ and $k_n < n^{d-2}/(10d)$. Then, by~\eqref{eq:sumi1} and~\eqref{eq:sumi2} we have that for every $i\in [n]$, $|k_i - k_{i-1}|\le 4n^{d-2}/(10d)$.
At the same time, by~\eqref{eq:sum k-s} there is $i\in [2,n]$ such that $k_i\ge n^{d-1}/(2d)$. Thus, both $k_2$ and $k_n$ must be at least $n^{d-1}/(2d) - n\times4 n^{d-2}/(10d) = n^{d-2}/(10d)$, which leads to a contradiction and finishes the proof.
\end{proof}
\begin{proof}[Proof of Theorem~\ref{th:degree:size}]
By Lemma~\ref{lem:manyvert} there are at least $n^{d-2}/(20d^2)$ authorised vertices in $D$. Moreover, since for any vertex $\bv\in\bbQ^d_n$, the number of vertices at (graph) distance at most 4 from $\bv$ is at most $1+2d+2d(2d-1)+2d(2d-1)^2+2d(2d-1)^3 < (2d)^4$, we may find $n^{d-2}/(320d^6)$ authorised vertices at distance at least 5 from each other in $\bbQ^d_n$. Then, the alternating cycles in the second neighbourhoods of these vertices ensured by Lemma~\ref{lem:authorised} are disjoint, and therefore may be switched independently of each other, thus proving the desired result. 
\end{proof}
\begin{proof}[Proof of Theorem~\ref{th:degree:size:hypercube}]
Since for any vertex $\bv\in\bbQ^d$, the number of vertices at (graph) distance at most 4 from $\bv$ is at most $1+d+d(d-1)+d(d-1)^2+d(d-1)^3 < d^4$, 
we may find $2^d/d^4$ vertices at distance at least 5 from each other in $\bbQ^d$. Then, the alternating cycles in the second neighbourhoods of these vertices ensured by Lemma~\ref{lem:authorised} are disjoint, and therefore may be switched independently of each other, thus proving the desired result.
\end{proof}

\section{Ergodicity of the high-dimensional hypercube: proof of Theorem~\ref{th:ergodicity:hypercube}}

First, let us briefly outline the proof strategy. Let $Q_1$ (resp.\ $Q_2$) be the hypercube containing all vertices whose last coordinate is 1 (resp.\ 2), and call an edge (or a dimer) in $\bbQ^d$ \emph{crossing} if it contains one vertex in $Q_1$ and one vertex in $Q_2$. For any given $d\ge 2$, we fix a dimer configuration on $\bbQ^d$ and iteratively decrease the number of crossing edges until none are left, so that we can apply induction on $d$. At each step, the idea is to find an alternating cycle of length at most $4d-4$ in which all edges between $Q_1$ and $Q_2$ are dimers. To do this, we combine several technical lemmas related to the expansion properties of the hypercube with the analysis of a suitable exploration procedure of $Q_1$ and $Q_2$ along the alternating cycles in $\bbQ^d$ that do not use crossing edges which are not dimers.

\subsection{Preliminaries on isoperimetric inequalities in the hypercube}
\label{subsec:isoperimetry}
For a set $A\subseteq \bbQ^d$, we denote by $\partial A$ the set of vertices at (graph) distance 1 from $A$ in $\bbQ^d$. The next result is a version of a classical isoperimetric inequality of Harper~\cite{Harper66} (see also~\cites{Frankl81, Katona75, Keevash20}) which states that sets of given size with smallest vertex boundary are essentially balls.

\begin{theorem}[\cite{Harper66}, or also Lemma~5 and Theorem~2 in~\cite{Katona75}]\label{thm:Harper}
Fix positive integers $d$ and $a < 
2^d$. Then, there exists a unique choice of integers $1\le t\le k+1\le d$ and $t\le a_t<\dots<a_k<d$ such that 
\[a = \sum_{j=k+1}^d \binom{d}{j} + \sum_{i=t}^k \binom{a_i}{i}.\]
Moreover, 
\[\min\left\{|A\cup \partial A|: A\subseteq \bbQ^d
, |A|=a\right\} = \phi_d(a) := \sum_{j=k}^{d} \binom{d}{j} + \sum_{i=t}^k \binom{a_i}{i-1}.\]
{The same holds trivially with $\phi_d(0):=0$ and $\phi_d(2^d):=2^d$.}
\end{theorem}

We will show the following technical property of the function $\phi_d$.

\begin{lemma}\label{lem:technical}
Fix $d\ge 1$ and $k\in [0,d]$, and let $\phi = \phi_{d}$ from Theorem~\ref{thm:Harper}. Let $l_1,l_2\in[2^d]$ and $l=l_1+l_2$.
\begin{enumerate}
    \item If $l\le 2^{d}$, then $\phi(l_1)+\phi(l_2)\ge \phi(l)$.
    \item If $l > 2^{d}$, then $\phi(l_1)+\phi(l_2)\ge 2^{d} + \phi(l-2^{d})$.
\end{enumerate}
\end{lemma}
\begin{proof}
We prove both statements simultaneously. If some of $l_1$ and $l_2$ is $2^d$, there is nothing to prove. Without loss of generality, suppose that $1\le l_1\le l_2\le 2^{d}-1$. Let $X_1$ and $X_2$ be two subsets of $\bbQ^{d}$ with $|X_1| = l_1$, $|X_2| = l_2$, $|X_1\cup \partial X_1| = \phi(l_1)$, $|X_2\cup \partial X_2| = \phi(l_2)$ and $X_1\not\subseteq X_2$ (the last can be ensured by applying a suitable isomorphism of $\bbQ^d$ to $X_1$, if needed). Define $Y_1 = X_1\cup X_2$ and $Y_2 = X_1\cap X_2$. We claim that
\begin{equation}\label{eq:shift}
|Y_1\cup \partial Y_1|+|Y_2\cup \partial Y_2|\le |X_1\cup \partial X_1|+|X_2\cup \partial X_2|.    
\end{equation}
Indeed, let $\bw\in Y_1\cup \partial Y_1$. Then, $\bw$ belongs to at least one of $X_1\cup \partial X_1$ and $X_2\cup \partial X_2$. Moreover, if $\bw$ also belongs to $Y_2\cup \partial Y_2$, then $\bw$ belongs to both $X_1\cup \partial X_1$ and $X_2\cup \partial X_2$, thus showing~\eqref{eq:shift} after summation over all vertices in $\bbQ^{d}$.

In particular, this shows that, for every pair of integers $l_1,l_2$ satisfying $1\le l_1\le l_2\le 2^{d}-1$, there are integers $m_1\in [l_2+1,2^d]$ and $m_2\in [0,l_1-1]$ with $m_1+m_2 = l_1+l_2$ such that $\phi(m_1)+\phi(m_2) \le \phi(l_1)+\phi(l_2)$. Iterating this observation finishes the proof of both points.
\end{proof}

In its present form, Theorem~\ref{thm:Harper} is not useful for our purposes. Indeed, after a certain number of steps along an alternating path ending with a dimer, we wish to be able to reach many new vertices with a non-dimer edge. For this reason, we are particularly interested in the size of the vertex boundary $\partial A$ rather than the size of the closed neighbourhood $A\cup\partial A$ and only need to consider sets $A$ with the correct parity.
Call a vertex in $\bbQ^d$ \emph{even} if the sum of its coordinates is even, and \emph{odd} otherwise. We also say that a subset of vertices of $\bbQ^d$ is even (resp.\ odd) if it contains only even (resp.\ odd) vertices. 
The next lemma provides a lower bound for the size of the vertex boundary $\partial A$ of an even set $A\subseteq \bbQ^d$ in terms of the size
of the closed neighbourhood $B\cup \partial B$ for arbitrary sets $B\subseteq \bbQ^{d-1}$ of the same size as $A$, to which Theorem~\ref{thm:Harper} can be applied.

\begin{lemma}\label{lem:even}
Fix an integer $d\ge 2$ and an even set $A\subseteq \bbQ^d$. Then, 
\[|\partial A|\ge \min\left\{|B\cup \partial B|: B\subseteq \bbQ^{d-1},\, |B| = |A|\right\}.\]
\end{lemma}
\begin{proof}
Define the map $\pi: \bbQ^d\to\bbQ^{d-1}:(v_1, \ldots, v_d)\mapsto (v_1, \ldots, v_{d-1})$ and consider a vertex $\bu$ in the first neighbourhood of $\pi(A)$ in $\bbQ^{d-1}$. If $\bu\in \pi(A)$, then there are two distinct vertices $\bu^1, \bu^2\in \bbQ^d$ with $\pi(\bu^1) = \pi(\bu^2) = \bu$ and $\bu^2\in A$. Thus, $\bu$ is the image of the odd vertex $\bu^1\in \partial A$. Moreover, if $\bu\in \partial\, \pi(A)$, then there is a vertex $\bv\in A$ such that $\pi(\bv)$ is a neighbour of $\bu$. Thus, $\bu$ is the image of a neighbour of $\bv$ under $\pi$, which is an odd vertex in $\partial A$. Combining the two observations above shows that $|\partial A| \ge |\pi(A)\cup \partial\, \pi(A)|$, which finishes the proof.
\end{proof}

\subsection{Proof of Theorem~\ref{th:ergodicity:hypercube}}
Recall that $Q_1$ and $Q_2$ are the two halves of $\bbQ^d$ with last coordinates $1$ and $2$ respectively. Fix a dimer configuration $D$ in $\bbQ^d$, and suppose that $\bv$ is an even vertex in $Q_1$ contained in a crossing dimer. We define the graph $\Gamma$ with vertex set $[2]^d$ and edge set $E(Q_1)\cup E(Q_2)\cup D$ where $E(G)$ denotes the edge set of a graph $G$. In particular, $D$ may naturally be seen as a dimer configuration on $\Gamma$. 

For every integer $k\ge 0$, denote by $E_k$ (resp.\ $O_k$) the set of even (resp.\ odd) vertices that may be reached from $\bv$ by following an alternating path in $\Gamma$ of length at most $2k$ (resp.\ $2k+1$) starting with a non-dimer edge. In particular, $E_0 = \{\bv\}$ and $O_0 = Q_1\cap\partial\{\bv\}$. Also, we define $E_{k,i} = E_k\cap Q_i$ for $i\in[2]$. Note that for every integer $k\ge 0$,
\begin{equation}\label{eq:Es}
|E_{k+1}|\ge |O_k|\ge |\partial E_{k,1}\cap Q_1| + |\partial E_{k,2}\cap Q_2|.    
\end{equation}

We next aim to recursively prove a lower bound on $|E_k|$ as a function of $k$.
Define the sequences $(a_k)_{k\ge 0}$ and $(b_k)_{k\ge 0}$ by setting $a_0 = 1$, $b_0 = 2^{d-2}+1$ and, for every integer $k\ge 0$, $a_{k+1}$ (resp.\ $b_{k+1}$) is the minimum of $|\partial X_1\cap Q_1|+|\partial X_2\cap Q_2|$ over all couples 
of even sets $(X_1,X_2)$ such that $X_1\subseteq Q_1$, $X_2\subseteq Q_2$ and $|X_1|+|X_2| = a_k$ (resp.\ $|X_1|+|X_2| = b_k$). Note that by~\eqref{eq:Es}, for every integer $k\ge 0$, $|E_k|\ge a_k$, and if $|E_m|\ge b_0$ for some $m$, then for every integer $k\ge 0$, $|E_{k+m}|\ge b_k$. Before we use these lower bounds on $|E_k|$, let us first use the preliminaries from Section~\ref{subsec:isoperimetry} to prove that their last terms are large enough.

\begin{lemma}\label{cor:a_2d}
We have $a_{d-2} = 2^{d-2}$ and $b_{d-2} = 2^{d-1}$.
\end{lemma}
\begin{proof}
If $d=2$, the statement is trivial. Suppose that $d\ge 3$. We show by induction that for every integer $k\in [0,d-2]$, $\min(a_k,b_k-2^{d-2})\ge \sum_{j=d-2-k}^{d-2} \tbinom{d-2}{j}$. The proof for $b_{k}-2^{d-2}$ being identical, we only prove the inequality for $a_{k}$. The statement is trivially satisfied for $k=0$. Suppose that for some $k\in [d-2]$, the statement holds for $k-1$. Fix even sets $X_1\subseteq Q_1$ and $X_2\subseteq Q_2$ of size resp.\ $l_1$ and $l_2$ such that $l = l_1+l_2=\sum_{j=d-1-k}^{d-2} \tbinom{d-2}{j}\le a_{k-1}$.
Denoting $\phi = \phi_{d-2}$ from Theorem~\ref{thm:Harper}, and applying Theorem~\ref{thm:Harper} with $d-2$ instead of $d$ and Lemma~\ref{lem:even} with $d-1$ instead of $d$, we obtain that
\begin{equation}\label{eq:ineqs}
|\partial X_1\cap Q_1|\ge \phi(l_1)\quad \text{and}\quad |\partial X_2\cap Q_2|\ge \phi(l_2).
\end{equation}
Thus, Lemma~\ref{lem:technical} shows that 
\[|\partial X_1\cap Q_1| + |\partial X_2\cap Q_2|\ge \phi(l_1)+\phi(l_2)\ge \phi(l)=\sum_{j=d-2-k}^{d-2}\binom{d-2}{j},\] which completes the induction.
\end{proof}

\begin{proof}[Proof of Theorem~\ref{th:ergodicity:hypercube}]
Fix an integer $d\ge 2$ and any dimer configuration $D$ on $\bbQ^d$. We show that by switching along alternating cycles of length at most $4d-4$, we can reach the configuration where every dimer $\bu\bv$ satisfies that $u_1\neq v_1$, that is, all dimers are parallel to the first dimension (note that only one dimer configuration has this property, so this fact directly implies the claimed connectivity of $\cD_{2d-2}(\bbQ^d)$).

We show this by induction on the dimension. The base case is clear. Fix an integer $d\ge 3$ and suppose that the statement holds for $d-1$. If $D$ contains no crossing dimers, then it consists of a dimer configuration on $Q_1$ and a dimer configuration on $Q_2$. Then, the conclusion follows from the induction hypothesis for $d - 1$. 

Now, suppose that $D$ contains a crossing dimer. Fix an even vertex $\bv\in Q_1$ contained in a crossing dimer. We show that $|E_{2d-3}| = 2^{d-1}$. 
Note that, by Lemma~\ref{cor:a_2d}, $E_{d-2}$ has size at least $2^{d-2}$. If the inequality is strict, then $|E_{d-2}|\ge b_0$ and consequently $|E_{2(d-2)}|\ge b_{d-2} = 2^{d-1}$.

Suppose that $|E_{d-2}| = 2^{d-2}$. We show that $|O_{d-2}|\ge 2^{d-2}+1$.
If $E_{d-2,1}$ and $E_{d-2,2}$ are both non-empty, we show that the sizes of the vertex boundaries of both $E_{d-2,1}$ in $Q_1$ and of $E_{d-2,2}$ in $Q_2$ must be resp.\ (strictly) larger than $E_{d-2,1}$ and $E_{d-2,2}$. Indeed, for $i\in[2]$, the number of edges between $E_{d-2,i}$ and its vertex boundary (in $Q_i$) is $(d-1)|E_{d-2,i}|$, and at the same time this number is at most $(d-1)|\partial E_{d-2,i}\cap Q_i|$. Moreover, equality holds only if $\partial E_{d-2,i}\cap Q_i$ is not adjacent to any vertex outside $E_{d-2,i}$ in $Q_i$, which may only happen when $E_{d-2,i}$ contains all even vertices in $Q_i$, which in our case shows that $|\partial E_{d-2,i}\cap Q_i| > |E_{d-2,i}|$. 
Now, if $E_{d-2,2} = \varnothing$, say, then $E_{d-2}$ must contain all even 
vertices in $Q_1$. However, as there is a crossing dimer with an even vertex in $Q_1$, there is also one with an odd vertex in $Q_1$ and therefore in $O_{d-2}$. Hence, $|E_{d-1,1}|\ge |E_{d-2,1}| = 2^{d-2}$ and $E_{d-1,2}\neq\varnothing$, so $|E_{d-1}|\ge 2^{d-2}+1$ and an application of Lemma~\ref{cor:a_2d} shows that $|E_{2d-3}|\ge b_{d-2} = 2^{d-1}$.

Let $\bv'$ be the odd vertex of the crossing dimer containing $\bv$. Since $E_{2d-3}$ contains all even vertices in $\bbQ^d$, it contains a neighbour of $\bv'$ in $Q_2$, so $\bv'\in O_{2d-3}$. Hence, there is an alternating cycle in $\Gamma$ of length at most $4d-4$ containing the dimer $\bv\bv'$, whose switching decreases the number of crossing dimers. Iterating the above approach leaves no crossing dimers eventually, and thus finishes the proof of the connectivity of $\cD_{2d-2}(\bbQ^d)$. 

To show the bound on the diameter, note that every switching decreases the number of crossing dimers by at least two, so $2^{d-2}$ steps are sufficient to make all crossing dimers (in a fixed dimension) disappear. Since the induction above consists of $d-1$ steps, the distance (in $\mathcal D_{2d-2}(\bbQ^d)$) from any dimer configuration $D$ to the configuration where all dimers are parallel to the first dimension is at most $(d-1)2^{d-2}$, and therefore the diameter of $\mathcal D_{2d-2}(\bbQ^d)$ is at most twice as large.
\end{proof}

\section{Diameter lower bound: proof of Theorem~\ref{th:diameter}}
For this section we fix $n$ even and $\ell,d\ge 2$.
As above, we call a site $\bv\in\bbQ_n^d$ \emph{even} if $\sum_{i\in[d]}v_i$ is even and \emph{odd} otherwise. Given a dimer configuration $D$, we define a colouring of the vertices of $\bbQ_n^d$ in two colours (red and blue) as follows. Let $\bu\bv\in D$ be a dimer with $\bu$ odd and $\bv$ even. Let $i\in[d]$ be such that $|u_i-v_i|=1$. We colour both $\bu$ and $\bv$ \emph{red} if $u_i-v_i=1$, and \emph{blue} if $u_i-v_i=-1$.

\begin{example}[Pyramid configuration]
\label{ex:pyramids}
As an example, let us consider the \emph{pyramid configuration} of Figure~\ref{fig:pyramids} defined formally as follows. For each $\bv'\in\bbQ_{n}^{d-2}$ and even $\bv\in\bbQ_n^d$ such that $\bv=(v_1,v_2,\bv')$, the second vertex in the dimer of $\bv$ is:
\[\begin{cases}
(v_1+1,v_2,\bv')&\text{if }v_1<v_2\text{ and }v_1\ge n+1-v_2,\\
(v_1,v_2+1,\bv')&\text{if }v_1\le v_2\text{ and }v_1< n+1-v_2,\\
(v_1-1,v_2,\bv')&\text{if }v_1>v_2\text{ and }v_1\le n+1-v_2,\\
(v_1,v_2-1,\bv')&\text{if }v_1\ge v_2\text{ and }v_1>n+1-v_2.
\end{cases}\]
In terms of colouring, the dimers corresponding to the former two cases are red, while the remaining ones are blue. In particular, all sites $\bu\in\bbQ_n^d$ with $u_1<u_2$ are red, while those with $u_1>u_2$ are blue.
\end{example}
Theorem~\ref{th:diameter} will follow easily from Example~\ref{ex:pyramids} and the following observation.
\begin{lemma}
\label{lem:colours}
For any dimer configuration on $\bbQ_n^d$ there are exactly $n^d/2$ red vertices.\end{lemma}
\begin{proof}
Fix $i\in[d]$ and $j\in[n]$. Consider the dimers $\bu\bv$ such that $u_i=j$ and $v_i=j+1$. Since the numbers of even and odd sites $\bw$ with $w_i\le j$ are equal, the number of such dimers with $\bu$ even is equal to the number of such dimers with $\bu$ odd. Since these two types of dimers have different colours (red and blue, respectively) and each dimer is considered for exactly one choice of $(i,j)$, the result follows.
\end{proof}
\begin{proof}[Proof of Theorem~\ref{th:diameter}]
It is clear that switching a cycle cannot modify the colours of vertices outside the cycle. Therefore, by Lemma~\ref{lem:colours}, switching any cycle can only alter the colours of the sites within the cycle without changing the amount of sites of either colour. Consider the pyramid dimer configuration $D$ from Example~\ref{ex:pyramids} and its inverse $\bar D$ obtained by using the configuration in Figure~\ref{subfig:left} on even sections and Figure~\ref{subfig:right} on odd ones instead, which also leads to exchanging the two colours. Thus, all sites $\bu\in\bbQ_n^d$ with $u_1<u_2$ are red in $D$ and blue in $\bar D$. Among these sites, for each $i\in[n-1]$, there are $(n-i)n^{d-2}$ at distance $i$ from the complement of this set. Since each switching moves at most $\ell$ red sites at graph distance at most $\ell$, in order to reach $\bar D$ from $D$, we need to switch at least 
\[\frac{n^{d-2}}{\ell^2}\sum_{i=1}^{n-1}(n-i)i=\frac{n^{d-1}(n^2-1)}{6\ell^2}\]
alternating cycles.
\end{proof}

\section{Ergodicity on \texorpdfstring{$\bbT_{m,n}$}{Tmn}: proof of Theorem~\ref{th:triangular}}

In this section, we fix positive integers $m,n$ with $m$ even. The proof of Theorem~\ref{th:triangular} proceeds by induction by showing that, starting with any dimer configuration on $\bbT_{m,n}$, we can make all dimers on the lower boundary horizontal. Thereby, one can directly apply the induction hypothesis to the remaining box $(0,1)+\bbT_{m,n-1}$. The base of the induction consists of the cases $n\in\{1,2\}$. The first one is trivial (there is one dimer configuration), and the second one boils down to switching 4-cycles on $\bbQ_{(m,2)}^2$ (which is easily seen to be ergodic with diameter $m-1$) due to absence of diagonal dimers. We may therefore assume that $n\ge 3$.

\begin{figure}
    \centering
\begin{tikzpicture}[x=1cm,y=1cm]
    \draw[gray,very thin] (2,1) grid (5,3);
    \draw[gray,very thin] (1,2)--(1,3)--(2,3);
    \draw[gray,very thin] (2,1)--(1,2);
    \draw[gray,very thin] (3,1)--(1,3);
    \draw[gray,very thin] (4,1)--(2,3);
    \draw[gray,very thin] (5,1)--(3,3);
    \draw[gray,very thin] (5,2)--(4,3);
    \draw[very thick] (1,1)--(1,2) node[above right] {$\bx$};
    \draw[very thick](2,2)--(2,1) node[above right]{$\by$};
    \draw[very thick,blue,dashed] (2,2)--(1,2);
    \draw[very thick,blue,dashed] (1,1)--(2,1);
\end{tikzpicture}
\quad
\begin{tikzpicture}[x=1cm,y=1cm]
    \draw[gray,very thin] (1,2) grid (5,3);
    \draw[gray,very thin] (2,1)--(1,1)--(1,2);
    \draw[gray,very thin] (3,1)--(5,1)--(5,2);
    \draw[gray,very thin] (4,1)--(4,2);
    \draw[gray,very thin] (2,1)--(1,2);
    \draw[gray,very thin] (3,1)--(1,3);
    \draw[gray,very thin] (4,1)--(2,3);
    \draw[gray,very thin] (5,1)--(3,3);
    \draw[gray,very thin] (5,2)--(4,3);
    \draw[very thick] (1,1)--(1,2) node[above right] {$\bx$};
    \draw[very thick](2,1)--(3,1);
    \draw[very thick](2,2)--(3,2) node[above right]{$\by$};
    \draw[very thick,blue,dashed] (2,1)--(2,2);
    \draw[very thick,blue,dashed] (3,1)--(3,2);
\end{tikzpicture}
\quad
\begin{tikzpicture}[x=1cm,y=1cm]
    \draw[gray,very thin] (2,1) grid (5,3);
    \draw[gray,very thin] (1,3)--(2,3);
    \draw[gray,very thin] (1,2)--(2,2);
    \draw[gray,very thin] (2,1)--(1,2);
    \draw[gray,very thin] (4,1)--(2,3);
    \draw[gray,very thin] (5,1)--(3,3);
    \draw[gray,very thin] (5,2)--(4,3);
    \draw[very thick] (1,1)--(1,2) node[above right] {$\bx$};
    \draw[very thick](2,1)--(3,1);
    \draw[very thick](2,2)--(1,3) node[above right]{$\by$};
    \draw[very thick,blue,dashed] (1,1)--(2,1);
    \draw[very thick,blue,dashed] (3,1)--(2,2);
    \draw[very thick,blue,dashed] (1,3)--(1,2);
\end{tikzpicture}
\\
\vspace{1em}
\begin{tikzpicture}[x=1cm,y=1cm]
    \draw[gray,very thin] (1,1) grid (2,3);
    \draw[gray,very thin] (3,1) grid (5,3);
    \draw[gray,very thin] (2,1)--(3,1);
    \draw[gray,very thin] (2,1)--(1,2);
    \draw[gray,very thin] (3,1)--(1,3);
    \draw[gray,very thin] (4,1)--(2,3);
    \draw[gray,very thin] (5,1)--(3,3);
    \draw[gray,very thin] (5,2)--(4,3);
    \draw[very thick] (1,1)--(1,2) node[above right] {$\bx$};
    \draw[very thick](2,1)--(3,1);
    \draw[very thick](2,2)--(2,3) node[above right]{$\by$};
    \draw[very thick](3,2)--(3,3) node[above right]{$\bz$};
    \draw[very thick,blue,dashed] (2,2)--(3,2);
    \draw[very thick,blue,dashed] (2,3)--(3,3);
\end{tikzpicture}
\quad
\begin{tikzpicture}[x=1cm,y=1cm]
    \draw[gray,very thin] (1,2) grid (5,3);
    \draw[gray,very thin] (4,1) grid (5,2);
    \draw[gray,very thin] (2,1)--(1,1)--(1,3)--(2,3);
    \draw[gray,very thin] (3,1)--(3,2);
    \draw[gray,very thin] (2,1)--(1,2);
    \draw[gray,very thin] (3,1)--(1,3);
    \draw[gray,very thin] (5,1)--(3,3);
    \draw[gray,very thin] (5,2)--(4,3);
    \draw[very thick] (1,1)--(1,2) node[above right] {$\bx$};
    \draw[very thick](2,1)--(3,1);
    \draw[very thick](2,2)--(2,3) node[above right]{$\by$};
    \draw[very thick](3,2)--(4,1) node[above right]{$\bz$};
    \draw[very thick,blue,dashed] (2,1)--(2,2);
    \draw[very thick,blue,dashed] (3,1)--(4,1);
    \draw[very thick,blue,dashed] (3,2)--(2,3);
\end{tikzpicture}
\quad
\begin{tikzpicture}[x=1cm,y=1cm]
    \draw[gray,very thin] (1,1) grid (5,3);
    \draw[gray,very thin] (2,1)--(1,2);
    \draw[gray,very thin] (3,1)--(1,3);
    \draw[gray,very thin] (3,2)--(2,3);
    \draw[gray,very thin] (4,2)--(3,3);
    \draw[gray,very thin] (5,2)--(4,3);
    \draw[very thick] (1,1)--(1,2) node[above right] {$\bx$};
    \draw[very thick](2,1)--(3,1);
    \draw[very thick](2,2)--(2,3) node[above right]{$\by$};
    \draw[very thick](3,2)--(4,2) node[above right]{$\bz$};
    \draw[very thick](4,1)--(5,1);
    \draw[very thick,blue,dashed] (4,1)--(3,2);
    \draw[very thick,blue,dashed] (5,1)--(4,2);
\end{tikzpicture}
    \caption{Illustration of Case 1 of the proof of Theorem~\ref{th:triangular}.}
    \label{fig:case1}
\end{figure}
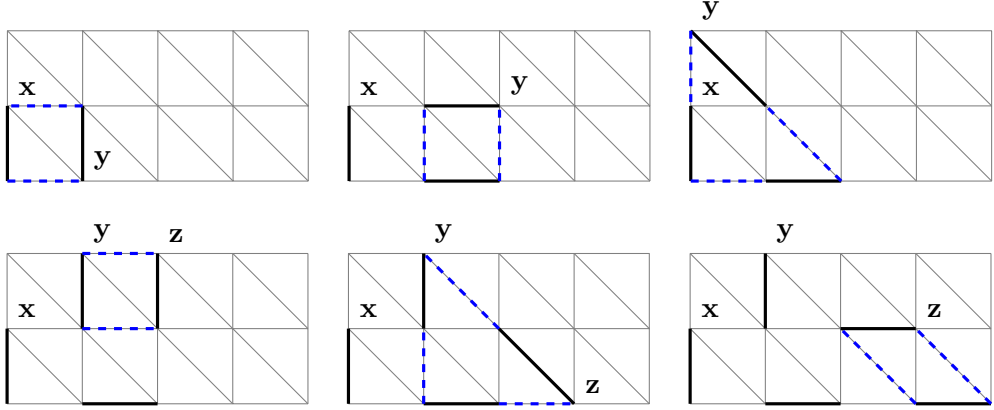

We proceed by a further induction. Let $x\in[m/2]$ be maximal such that $(2y-1,1),(2y,1)$ is a dimer in the present configuration for all $y\in[x]$. If $x=m/2$, we are done. Otherwise, we consider two cases for the other end $\bx$ of the dimer containing $(2x+1,1)$.

\paragraph{Case 1} Assume $\bx=(2x+1,2)$ (see Figure \ref{fig:case1} for an illustration). Let $\by$ be the other end of the dimer of $(2x+2,2)$. If $\by=(2x+2,1)$, we are done after switching the 4-cycle $(2x+1,1),(2x+1,2),(2x+2,2),(2x+2,1)$. We assume this is not the case, so necessarily $(2x+2,1),(2x+3,1)$ is a dimer. If $\by=(2x+3,2)$, then switching the 4-cycle $(2x+2,1),(2x+3,1),(2x+3,2),(2x+2,2)$ brings us to the previous case, so we are done. If $\by=(2x+1,3)$, then we are done by switching the 6-cycle $(2x+1,1),(2x+2,1),(2x+3,1),(2x+2,2),(2x+1,3),(2x+1,2)$. 

We may therefore assume that $\by=(2x+2,3)$, and note that it suffices to move this dimer from $(2x+2,2),(2x+2,3)$ since all other cases were already dealt with. Let $\bz$ be the other end of the dimer of $(2x+3,2)$. If $\bz=(2x+3,3)$, it suffices to switch the 4-cycle $(2x+2,2),(2x+3,2),(2x+3,3),(2x+2,3)$. If $\bz=(2x+4,1)$, it suffices to switch the 6-cycle $(2x+2,1),(2x+3,1),(2x+4,1),(2x+3,2),(2x+2,3),(2x+2,2)$. Finally, it remains to consider the case $\bz=(2x+4,2)$, which entails that $(2x+4,1),(2x+5,1)$ is a dimer. Then, we can switch the 4-cycle $(2x+4,1),(2x+5,1),(2x+4,2),(2x+3,2)$, returning to the previous case.

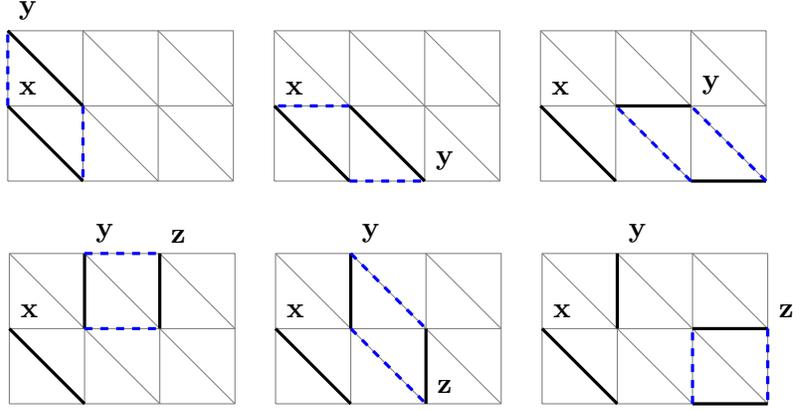
\begin{figure}
    \centering
\begin{tikzpicture}[x=1cm,y=1cm]
    \draw[gray,very thin] (1,2) grid (3,3);
    \draw[gray,very thin] (1,2)--(0,2)--(0,1)--(3,1)--(3,2);
    \draw[gray,very thin] (2,1)--(2,2);
    \draw[gray,very thin] (0,3)--(1,3);
    \draw[gray,very thin] (2,1)--(0,3);
    \draw[gray,very thin] (3,1)--(1,3);
    \draw[gray,very thin] (3,2)--(2,3);
    \draw[gray,very thin] (1,1)--(0,2);
    \draw[very thick] (1,1)--(0,2) node[above right] {$\bx$};
    \draw[very thick](1,2)--(0,3) node[above right]{$\by$};
    \draw[very thick,blue,dashed] (1,2)--(1,1);
    \draw[very thick,blue,dashed] (0,2)--(0,3);
\end{tikzpicture}
\quad
\begin{tikzpicture}[x=1cm,y=1cm]
    \draw[gray,very thin] (1,2) grid (3,3);
    \draw[gray,very thin] (2,1) grid (3,2);
    \draw[gray,very thin] (1,3)--(0,3)--(0,1)--(1,1)--(1,2);
    \draw[gray,very thin] (2,1)--(0,3);
    \draw[gray,very thin] (3,1)--(1,3);
    \draw[gray,very thin] (3,2)--(2,3);
    \draw[gray,very thin] (1,1)--(0,2);
    \draw[very thick] (1,1)--(0,2) node[above right] {$\bx$};
    \draw[very thick](1,2)--(2,1) node[above right]{$\by$};
    \draw[very thick,blue,dashed] (1,2)--(0,2);
    \draw[very thick,blue,dashed] (1,1)--(2,1);
\end{tikzpicture}
\quad
\begin{tikzpicture}[x=1cm,y=1cm]
    \draw[gray,very thin] (0,1) grid (3,3);
    \draw[gray,very thin] (1,2)--(0,3);
    \draw[gray,very thin] (2,2)--(1,3);
    \draw[gray,very thin] (3,2)--(2,3);
    \draw[gray,very thin] (1,1)--(0,2);
    \draw[very thick] (1,1)--(0,2) node[above right] {$\bx$};
    \draw[very thick](1,2)--(2,2) node[above right]{$\by$};
    \draw[very thick](2,1)--(3,1);
    \draw[very thick,blue,dashed] (2,1)--(1,2);
    \draw[very thick,blue,dashed] (3,1)--(2,2);
\end{tikzpicture}
\\
\vspace{1em}
\hspace{0.95em}
\begin{tikzpicture}[x=1cm,y=1cm]
    \draw[gray,very thin] (1,3)--(0,3)--(0,1)--(3,1)--(3,3)--(2,3);
    \draw[gray,very thin] (0,2)--(1,2)--(1,1);
    \draw[gray,very thin] (2,1)--(2,2)--(3,2);
    \draw[gray,very thin] (2,1)--(0,3);
    \draw[gray,very thin] (3,1)--(1,3);
    \draw[gray,very thin] (3,2)--(2,3);
    \draw[gray,very thin] (1,1)--(0,2);
    \draw[very thick] (1,1)--(0,2) node[above right] {$\bx$};
    \draw[very thick](1,2)--(1,3) node[above right]{$\by$};
    \draw[very thick](2,2)--(2,3) node[above right]{$\bz$};
    \draw[very thick,blue,dashed] (1,3)--(2,3);
    \draw[very thick,blue,dashed] (1,2)--(2,2);
\end{tikzpicture}
\quad
\begin{tikzpicture}[x=1cm,y=1cm]
    \draw[gray,very thin] (0,1) grid (3,3);
    \draw[gray,very thin] (1,2)--(0,3);
    \draw[gray,very thin] (3,1)--(2,2);
    \draw[gray,very thin] (3,2)--(2,3);
    \draw[gray,very thin] (1,1)--(0,2);
    \draw[very thick] (1,1)--(0,2) node[above right] {$\bx$};
    \draw[very thick](1,2)--(1,3) node[above right]{$\by$};
    \draw[very thick](2,2)--(2,1) node[above right]{$\bz$};
    \draw[very thick,blue,dashed] (1,2)--(2,1);
    \draw[very thick,blue,dashed] (1,3)--(2,2);
\end{tikzpicture}
\quad
\begin{tikzpicture}[x=1cm,y=1cm]
    \draw[gray,very thin] (0,2) grid (3,3);
    \draw[gray,very thin] (2,1)--(0,1)--(0,2);
    \draw[gray,very thin] (1,1)--(1,2);
    \draw[gray,very thin] (2,1)--(0,3);
    \draw[gray,very thin] (3,1)--(1,3);
    \draw[gray,very thin] (3,2)--(2,3);
    \draw[gray,very thin] (1,1)--(0,2);
    \draw[very thick] (1,1)--(0,2) node[above right] {$\bx$};
    \draw[very thick](1,2)--(1,3) node[above right]{$\by$};
    \draw[very thick](2,2)--(3,2) node[above right]{$\bz$};
    \draw[very thick](2,1)--(3,1);
    \draw[very thick,blue,dashed] (2,1)--(2,2);
    \draw[very thick,blue,dashed] (3,2)--(3,1);
\end{tikzpicture}
    \caption{Illustration of Case 2 of the proof of Theorem~\ref{th:triangular}.}
    \label{fig:case2}
\end{figure}
\paragraph{Case 2} Assume $\bx=(2x,2)$ (see Figure \ref{fig:case2} for an illustration). Let $\by$ be the other end of the dimer of $(2x+1,2)$. If $\by=(2x,3)$, then switching the 4-cycle $(2x+1,2),(2x,3),(2x,2),(2x+1,1)$ brings us back to Case 1. If $\by=(2x+2,1)$, then we are done by switching the 4-cycle $(2x+1,1),(2x+2,1),(2x+1,2),(2x,2)$. If $\by=(2x+2,2)$, then $(2x+2,1),(2x+3,1)$ has to be a dimer and switching the 4-cycle $(2x+2,1),(2x+3,1),(2x+2,2),(2x+1,2)$ returns us to the previous case. 

We may therefore assume that $\by=(2x+1,3)$ and it suffices to move this dimer to $(2x+1,2),(2x+2,2)$. Let $\bz$ be the other end of the dimer of $(2x+2,2)$. If $\bz\in\{(2x+2,1),(2x+2,3)\}$, this forms a 4-cycle with $(2x+1,2),(2x+1,3)$ and we are done. It therefore remains that $\bz=(2x+3,2)$, which forces $(2x+2,1),(2x+3,1)$ to be a dimer, and switching the 4-cycle formed by these two dimers returns us to the previous case. 

This completes the induction as well as the proof of the first statement in Theorem~\ref{th:triangular}. For the second statement, note that we switch at most 4 alternating cycles in the process of making the dimer at $(2x+1,1)$ horizontal.

\begin{remark}
Let us note that the proof entails that the minimum degree of $\cD_6(\bbT_{m,n})$ is at least linear in the semi-perimeter $m+n$. In contrast, the minimum degree of $\cD_8(\bbT_{m,n})$ can be shown to be of order $mn$ since there is an alternating cycle of length at most 8 within the third neighbourhood of each vertex. We also observe that ``sufficiently regular'' domains such as triangles or hexagons with even number of vertices can be treated along the lines of Theorem~\ref{th:triangular}.
\end{remark}

\section*{Acknowledgements}
This work was supported by the Austrian Science Fund (FWF): P35428-N. We thank Scott Sheffield and Catherine Wolfram for suggesting the topic to us, Marcin Lis for several interesting discussions and the anonymous referee for helpful suggestions on the presentation.

\bibliographystyle{plain}
\bibliography{Bib}

\end{document}